\documentclass[a4paper,12pt]{article}

\usepackage{typearea}
\typearea{15}

\usepackage[dvipdfmx]{graphicx}
\usepackage{color}

\usepackage{mymacros2}
\newcommand{\Perv}{\mathop{\mathrm{Perv}}}
\title{Categorification of Legendrian knots}
\author{Tatsuki Kuwagaki}
\date{{\em dedicated to Saito Kyoji sensei on his 75th birthday}}
\begin{document}

\maketitle

\begin{abstract}
Perverse schober defined by Kapranov--Schechtman is a categorification of the notion of perverse sheaf. In their definition, a key ingredient is certain purity property of perverse sheaves. In this short note, we attempt to describe a real analogue of the above story, as categorification of Legendrian points/knots. The notion turns out to include various notions such as semi-orthogonal decomposition, mutation braiding, spherical functor, N-spherical functor, and irregular perverse schober.
\end{abstract}

\section{Perverse schober}
{\em Perverse schober} is a categorification of the notion of perverse sheaf, found by Kapranov--Schechtman~\cite{KapS}. In this section, let us recall their observations over a one-punctured disk briefly.

Let $\bD$ be a standard open disk in $\bC$ centered at $0$. We will consider the category of perverse sheaves with singularity at $0$ and denote it by $\Perv(\bD,0)$. The category is known to have the following linear-algebraic description: Let $\cC$ be the category given by the following data:
\begin{enumerate}
\item Object: a pair of vector spaces $(V,W)$ with a pair of linear maps $f\colon V\rightarrow W$ and $g\colon W\rightarrow V$ satisfying $\id_V-fg$ and $\id_W-gf$ are invertible.
\item Morphism: compatible linear maps.
\end{enumerate}

\begin{theorem}[Beilinson~\cite{Beilinson}]
There exists an equivalence between $\cC$ and $\Perv(\bD,0)$.
\end{theorem}
For a given perverse sheaf, the two vector spaces are given by the space of vanishing cycles and nearby cycles, or more explicitly, $V:=\RGamma_L(\cE)_0$ and $W:=\RGamma_L(\cE)_x$ for a perverse sheaf $\cE$ where $L$ is the interval inside the disk $\bD$ (Figure 1.1) and $\RGamma_L(\cdot)$ is the local cohomology sheaf. 
\begin{figure}[h]
\begin{center}
  \includegraphics[width=5cm]{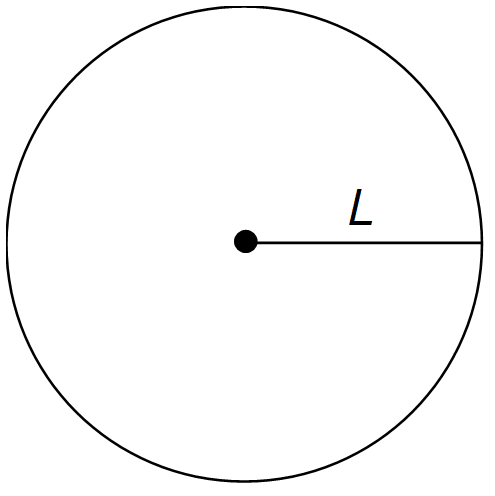}
\end{center}
\caption{skeleton $L$}
\end{figure}

   Even though a perverse sheaf is a complex of sheaves, its vanishing cycles and nearby cycles are vector spaces with a single degree. This {\em purity} property enables Kapranov--Schechtman to consider a categorification of perverse sheaves even with the lack of the definition of ``complexes of categories''.

They define a categorification in the following way. The data is the following: two stable dg-categories $\cC$ and $\cD$, a pair of adjoint functors $F\colon \cC\rightarrow \cD$, the left adjoint $F^L\colon \cD\rightarrow \cC$, and the right adjoint $F^R\colon \cD\rightarrow \cC$ satisfying $\Cone(\id_{\cC}\rightarrow F^RF)$ and $\Cone(FF^R\rightarrow \id_{\cD})$ are autoequivalences. Then the induced morphisms between the Grothendieck groups $K_0(\cC)\otimes_\bZ\bC$ and $K_0(\cD)\otimes_\bZ\bC$ gives a perverse sheaf by Beilinson's theorem. Hence this notion is actually a categorification of perverse sheaf and it turns out that this notion was previously known as a spherical functor by Anno--Logvinenko~\cite{AL}. Kapranov--Schechtman considered speherical functor as one representation of categorification of perverse sheaf (``perverse schober'') over $\bD$ with one singularity. Actually, there are other realizations if we choose other skeletons like in Figure 1.2. 
\begin{figure}[h]
\begin{center}
  \includegraphics[width=5cm]{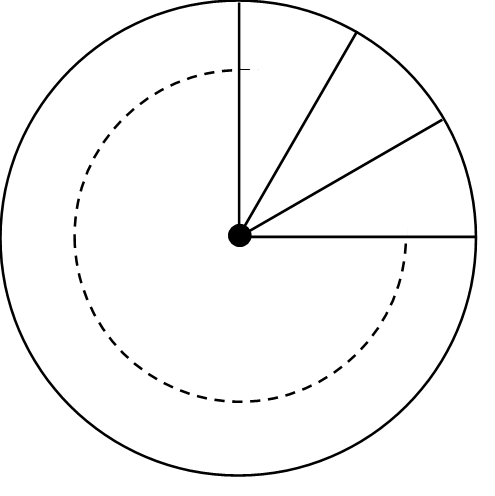}
\end{center}
\caption{$A_n$-skeleton}
\end{figure}

For example, if we have $n=2$, then this gives a notion of {\em spherical pair}, which is also a categorification of perverse sheaves. Also,
they can be defined over general surfaces with arbitrary number of singular points.

There are many interesting examples of perverse schobers coming from VGIT wall-crossing \cite{Donovan1, Donovan2}, Flops \cite{BKS}, mirror symmetry \cite{Nadler, DonovanKuwagaki, HarKat}.

\section{Purity in microlocal sheaf theory}
Next, we would like to describe a real analogue. Let $M$ be either $\bR$ or $\bR^2$. Let $C$ be a compact manifold with dimension equals to $\dim M-1$ (possibly with multiple connected components) and $\iota\colon C \rightarrow M$ be an immersion. Then the conormal bundle of $S:=\iota(C)$ has two components over each component of $C$. We choose one component of the conormal bundle over each component of $C$, which we say a choice of {\em co-orientation}.

The co-orientation is a conical Lagrangian subset $L:=L_C$ of $T^*M$. It is the same as the data of Legendrian point/knot $K:=K_L$ at contact infinity of $T^*M$ i.e. $L=\bR_{>0}\cdot K$. We set $L(K):=L\cup T^*_MM$ where  $T^*_MM$ is the zero section.

We would like to consider a (weakly) constructible sheaf $\cE$ over $M$ whose microsupport satisfies $\SS(\cE)\subset L(K)$. For readers who are not familiar with the notion of microsupport defined by Kashiwara--Schapira \cite{KS}, we would like to explain it in some plain words.

For simplicity, we further assume that the cardinality of each fiber of $\iota$ is at most two.
\begin{condition}
\begin{enumerate}
\item Let $S_{sm}$ be the smooth locus of $S=\iota(C)$ and $S_{sing}$ be the singular locus. Then we have a decomposition $M=S_{sm}\sqcup S_{sing}\sqcup (M\bs S)$. Then the first condition is that a sheaf valued in $\bC$-vector spaces $\cE$ is constructible with respect to this decomposition i.e. For each stratum $\sigma$ of the decomposition, the restriction $\cE|_\sigma$ is a locally constant sheaf.

\item Let us take $p\in K$, in other words, let us pick a ray (a single orbit of the $\bR_{>0}$-action) in $L$ with the condition $x:=\pi(p)\in S_{sm}$ where $\pi\colon T^*M\rightarrow M$ is the projection. Take a small neighborhood $U$ of $x$ such that $U\cap S\subset S_{sm}$ and $U\bs S_{sm}$ has exactly two components (Figure \ref{fig2.1}). The one of two components of $U\bs S_{sm}$ is denoted by $U_+$ if it is in the direction of $p$. Other one is denoted by $U_-$. Then the second condition asks that the restriction morphism $\RGamma(U, \cE)\rightarrow \RGamma(U_+,\cE)$ is a quasi- isomorphism.
\end{enumerate}
\end{condition}

\begin{figure}[h]
\begin{center}
  \includegraphics[width=5cm]{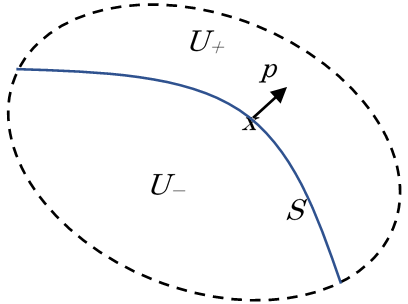}
\end{center}
\caption{Defining microsupport}\label{fig2.1}
\end{figure}

Let us also pick $x_+\in U_+$ and $x_-\in U_-$. By the condition 1, the restriction map $\RGamma(U, \cE)\rightarrow \cE_{x}$ and $\RGamma(U_+, \cE)\rightarrow \cE_{x_+}$ are isomorphisms. So the condition 2 asks that the canonical morphism $\cE_x\rightarrow \cE_{x_+}$ is a quasi-isomorphism or not. This condition of course does not depend on the choice of $x_+$, $x_-$. Also, it is independent among the choice of $x$ inside one component of $S_{sm}$. There exists the following fact.

\begin{lemma}
Condition 2.1 is equivalent to $\SS(\cE)\subset L(K)$.
\end{lemma}
Hence one can consider Condition 2.1 as the definition.

In the above setup, we also have a canonical morphism $\cE_{x}\rightarrow \cE_{x_-}$. We set
\begin{equation}
\cE_p:=\Cone(\cE_{x}\rightarrow \cE_{x_-}),
\end{equation}
which is a priori a complex of vector spaces. This is called {\em microstalk} of $\cE$ at $p$.

The following definition is made by Kashiwara--Schapira \cite{KS}.
\begin{definition}
We say $\cE$ is {\em pure} if $\cE_p$ is concentrated in degree $0$ for all $p\in K$.
\end{definition}
We will use this purity to get a categorification of Legendrian points and knots in the following sections.

\section{Categorification of Legendrian points}
In this section, we would like to discuss the case of $M=\bR$. Then $C=\{x_{1/2},...,x_{n-1/2}\}$ is a finite set of points. With this notation, we mean $x_{i+1/2}$ is on the right of $x_{j+1/2}$ if $i>j$. Let us fix the co-orientation over $C$ which gives $L_C\subset T^*M$. Let $\bR\bs C=\sqcup_{i\in 0,..., n}J_i$ be the decomposition into the connected components where the boundaries of $J_i$ is $x_{i-1/2}$ and $x_{i+1/2}$. Let $J_i, J_{i+1}$ be the adjacent intervals i.e., the closures of them intersect. 

Let $\cE$ be a sheaf micro-supported in $L(K)$ and we assume it is pure. Take $y_i\in J_i$ and $y_{i+1}\in J_{i+1}$. 

Suppose that the co-orientation over $x_{i/2}$ is positive. By the discussion of the definition of microsupport, there exists an identification $\cE_{x_{i+1/2}}\cong \cE_{y_{i+1}}$ and we have a generalization map from $\cE_{x_{i+1/2}}$ to $\cE_{y_{i}}$. Combining these we have a map $f_{i+1/2}\colon \cE_{y_{i+1}}\rightarrow \cE_{y_{i}}$. If the co-orientation is negative, we get a map $f_{i+1/2}\colon \cE_{y_{i}}\rightarrow \cE_{y_{i+1}}$. By the purity, $\cE_{y_i}$, $\cE_{y_i+1}$, and the cone of these morphisms are all vector spaces (not complexes). This implies that $f_{i+1/2}$ is injective and the cone is the cokernel of $f_{i+1/2}$. Hence we have the following:

\begin{proposition}
The category of pure sheaves micro-supported in $L$ is equivalent to the category given by the following data:
\begin{enumerate}
\item Object: $(\{V_i\}_{i\in 0,...,n}, \{f_{i+1/2}\}_{i\in 0,...,n-1})$ where, for any $i$, $V_i$ is a finite-dimensional vector space, $f_{i+1/2}$ is an injective morphism from $V_i$ to $V_{i+1}$ if the co-orientation over $x_i$ is negative, $f_{i+1/2}$ is an injective morphism from $V_{i+1}$ to $V_{i}$ if the co-orientation over $x_i$ is positive,   
\item Morphism: compatible linear maps.
\end{enumerate}
\end{proposition}
So these sheaves are expressed in terms of very simple linear-algebraic data.

Let us consider the simplest case where $C$ is a singleton $C_*=\{x_1/2\}$ and has the negative co-orientation. Every situation is locally the same as this situation up to the inversion of the orientation.

\begin{ansatz}
A categorification $\frakC$ of $L_{C_*}$ is a triangulated category $\cC$ with a semi-orthogonal decomposition
\begin{equation}
\cC=\la \cC_0, \cC_1\ra.
\end{equation}
Then the stalk $\frakC_{y_i}$ over $y_i\in J_i$ is set by $\frakC_{y_i}:=\la \cC_0,\cC_i\ra$ and the microstalk $\frakC_{p_{1/2}}$ over $p_{1/2}$ with $\pi(p_{1/2})=x_{0}$ is set by $\cC_1$.
\end{ansatz}
Since we have the localization
\begin{equation}
\frakC_{y_0}\hookrightarrow \frakC_{y_1}\rightarrow \cC_{1},
\end{equation}
by taking the Grothendieck group ant tensor $\bC$ over $\bZ$, we get an exact sequence of $\bC$-vector spaces
\begin{equation}
K_0(\frakC_0)\otimes_\bZ\bC\hookrightarrow K_0(\frakC_{1})\otimes_\bZ\bC\rightarrow K_0(\cC_1)\otimes_\bZ\bC.
\end{equation}
This exact sequence gives a pure sheaf microsupported in $L_{C_*}$, hence the ansatz is justified.

From this ansatz, one can consider a categorification for any co-orientation of $C$. Let us consider the case where the co-orientation over each point in $C$ is negative. In this case, the data $\{V_i\}_{i=0,...,n}$ is a sequence of injective morphisms i.e., a filtered vector space indexed by $\{0,...,n\}$. Then in this case, a categorification $\fC$ is given by a triangulated category $\cC$ with a semi-orthogonal decomposition
\begin{equation}
\cC=\la \cC_0,...,\cC_{n}\ra.
\end{equation}
Then the stalk $\frakC_{y_i}$ over $y_i\in J_i$ is set by $\frakC_{y_i}:=\la \cC_0,..., \cC_i\ra$ and the microstalk $\frakC_{p_{i+1/2}}$ over $p_{i+1/2}$ with $\pi(p_{i+1/2})=x_{i+1/2}$ is set by $\cC_{i+1}$.

\section{Categorification of Legendrian knots}
In this section, let us consider the case $M=\bR^2$. Then $C$ is a curve in this case. To simplify the discussion, we assume that $\iota$ is an immersion which is an embedding up to finite transversal double points. We call these singular points of the immersion ``crossing points''.

\begin{remark}[Cusps]
In general, when we consider ``front projection'' for Legendrian knots, thery can have cusps.
In this note, we will avoid the appearance of cusps. In the presence of cusps, we can still talk about pure sheaves following Kashiwara--Schapira and we can still talk about their categorification by introducing a pair of a category and an integer which categorifies a shifted vector space. However we do not treat this notion in this note, since we do not have any interesting examples of this categorification.
\end{remark}

Let us consider a local picture around a crossing point (Figure \ref{fig4.1}).
\begin{figure}[h]
\begin{center}
  \includegraphics[width=5cm]{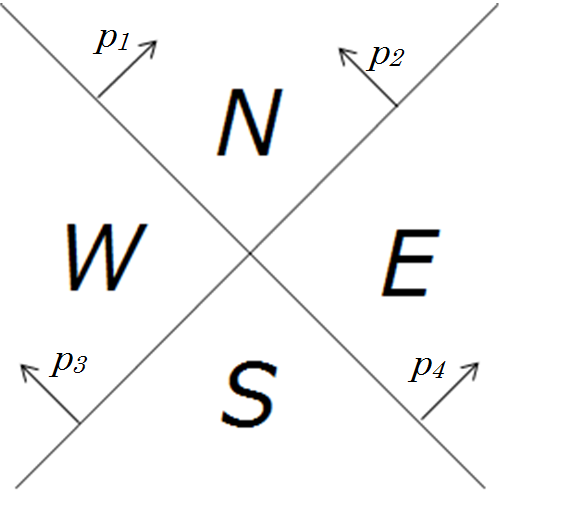}\label{fig4.1}
\end{center}
\caption{A crossing point with a coorientation}
\end{figure}
Here the arrows are indicating the co-orientations. Consider a pure sheaf $\cE$ micro-supported in $L_C$.  For $*\in \{N, E, W, S\}$, $\cE_{*}$ means the stalk of $\cE$ over a point in the corresponding domain indicated in Figure 4.1. Again we have morphisms, $\cE_N\rightarrow \cE_W,\cE_E$ and $\cE_W, \cE_E\rightarrow \cE_S$.

\begin{proposition}[\cite{STZ}]
The sequence
\begin{equation}
\cE_N\rightarrow \cE_W\oplus \cE_E\rightarrow \cE_S
\end{equation}
is exact.
\end{proposition}
Using this, we have the following.
\begin{proposition}[\cite{STZ}]
The category of pure sheaves micro-supported in a crossing point is equivalent to the category given by the following data:
\begin{enumerate}
\item Object: $(V_N, V_W, V_E, V_S, f_{NW}, f_{NE}, f_{WS}, f_{ES})$ where $V_{a}$ is a finite dimensional vector space and $f_{ab}\colon V_a\rightarrow V_b$ is a linear inclusion for any $a, b\in \{N, W, E, S\}$. Moreover, they satisfy the following;
a sequence 
\begin{equation}
\cE_N\xrightarrow{f_{NW}+f_{NE}} \cE_W\oplus \cE_E\xrightarrow{f_{WS}-f_{ES}} \cE_S
\end{equation}
is an exact sequence.
\item Morphism: compatible linear maps.
\end{enumerate}
\end{proposition}

Note that there exists a short exact sequence of complexes
\begin{equation}
0\rightarrow (\cE_N\rightarrow \cE_W)\rightarrow (\cE_N\rightarrow \cE_W\oplus \cE_E\rightarrow \cE_S)\rightarrow \cE_S/\cE_W\rightarrow 0. 
\end{equation}
Since the middle term is acyclic, we have a  quasi-isomorphism $\cE_W/\cE_N\cong \cE_S/\cE_W$.
Since $\cE_{p_1}\cong \cE_W/ \cE_N$ and $\cE_{p_3}\cong \cE_S/\cE_W$, this implies $\cE_{p_1}\cong \cE_{p_4}$. Similarly, one can deduce $\cE_{p_2}\cong \cE_{p_3}$. This is the locally constant property of microstalks \cite{KS}.

To consider a categorification of a crossing point, let us consider the two paths $\gamma_1, \gamma_2$ depicted in Figure 4.2. 

\begin{figure}[h]
\begin{center}
  \includegraphics[width=5cm]{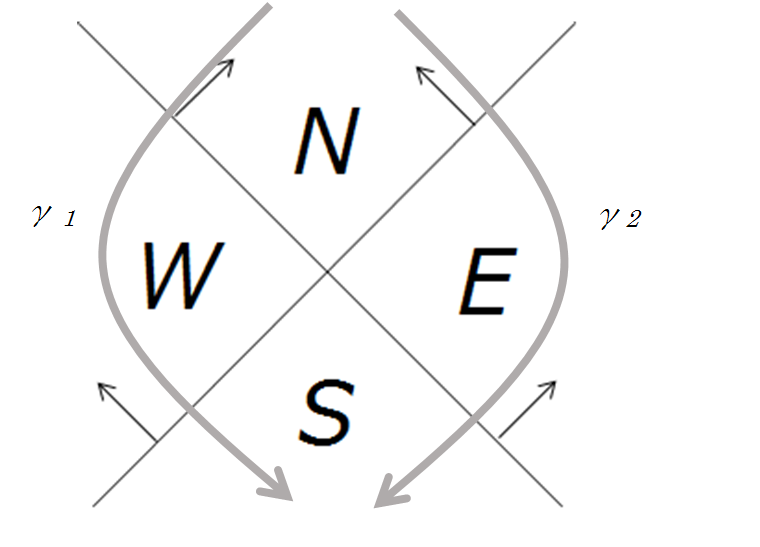}
\end{center}
\caption{Two paths}
\end{figure}
Then the pull back of a categorification of a crossing point along each $\gamma_i$ should be a categorification of two negative Legandrian points over $\gamma_i$. 

From these intuitions, we can imagine some necessary condition to categorify a crossing point. 
\begin{enumerate}
\item Over points $N, E, W, S$, stalks are triangulated categories $\fC_N, \fC_E, \fC_W, \fC_S$. 
\item We have semi-orthogonal decompositions $\fC_S=\la \fC_N, \cC_{11}, \cC_{12}\ra$ along $\gamma_1$ and $\fC_S=\la \fC_N, \cC_{21},\cC_{22}\ra$ along $\gamma_2$.
\item Micro-stalks can be considered as $\fC_{p_1}\cong \cC_{11}$, $\fC_{p_2}\cong \cC_{12}$, $\fC_{p_3}\cong \cC_{21}$, and $\fC_{p_4}\cong \cC_{22}$.
\end{enumerate}
Then the locally constant property of the micro-stalks, it is natural to assume $\cC_{11}\cong \cC_{{22}}$ and $\cC_{21}\cong \cC_{12}$. Hence, from $\gamma_1$ to $\gamma_2$, the semi-orthogonal components of $\fC_S$ are flipped; 
\begin{equation}
\la \fC_N, \cC_{11}, \cC_{12}\ra\rightsquigarrow \la \fC_N, \cC_{12}', \cC_{11}'\ra:= \la \fC_N, \cC_{21}, \cC_{22}\ra.
\end{equation}
To realize this relation naturally, we set the following ansatz.
\begin{ansatz}
A categorification $\frakC$ of a crossing point is triangulated categories $\cC$ and $\cC'$ with semi-orthogonal decompositions
\begin{equation}
\begin{split}
\cC&=\la \fC_N, \cC_{11}, \cC_{12}\ra, \\
\cC'&=\la \fC'_N, \cC_{12}', \cC_{11}'\ra.
\end{split}
\end{equation}
with an equivalence $\cC'\xrightarrow{f}\cC$ such that 
\begin{equation}
\la f(\fC'_N), f(\cC_{12}'), f(\cC_{11}')\ra=\la \fC_N, \bL_{C_{11}}\cC_{12}, \cC_{11}\ra
\end{equation}
as semi-orthogonal decompositions where the right hand side is the left mutation at $\cC_{11}$. Then the stalk are given by $\frakC_{W}:=\la \fC_N, \cC_{11}\ra$, $\frakC_E:=\la\fC_N, \cC_{12}'\ra$, and $\frakC_S:=\fC$. Microstalks are $\frakC_{p_1}:=\cC_{11}$, $\frakC_{p_2}:=\cC_{12}'$, $\frakC_{p_3}:=\cC_{12}$, and $\frakC_{p_4}:=\cC_{11}$,
\end{ansatz}
By taking $K_0(\bullet)\otimes_\bZ\bC$, we get a sheaf micro-supported in a crossing point.

\begin{example}
Let us describe a bit fancy example. Let $\cC$ be a triangulated category with an exceptional collection $\cC=\la E_1, ..., E_n\ra$. Then it is well-known that the braid group $\mathrm{Br}_n$ acts on the set of exceptional collections of $\cC$; let $\sigma_i$ be a positive braiding of $i$-th braid and $i+1$-th braid. Then a part of the exceptional collection $\la E_i, E_{i+1}\ra$ is mutated into $\la E_{i+1}', E_i\ra$. 

For a positive braid $\sigma$, we can associate a Lengendrian $K$. 
\begin{figure}[h]
\begin{center}
  \includegraphics[width=5cm]{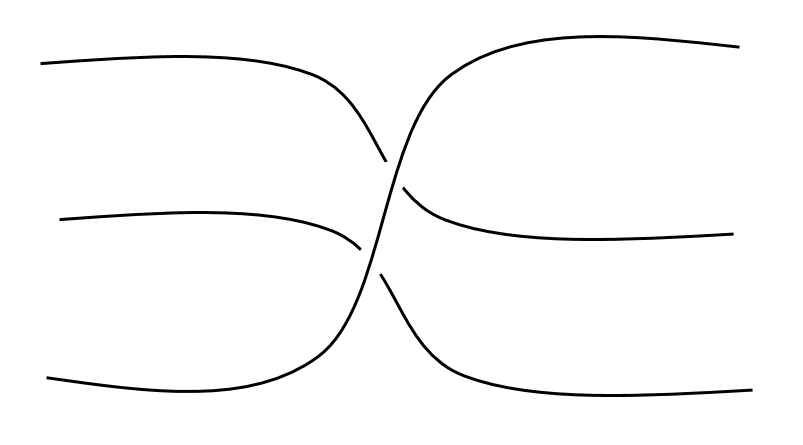}
\end{center}
\caption{Braid}
\end{figure}

\begin{figure}[h]
\begin{center}
  \includegraphics[width=5cm]{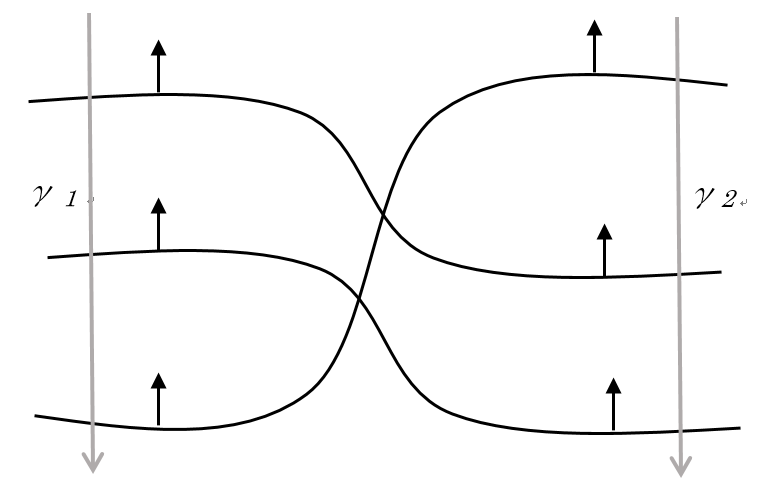}
\end{center}
\caption{Legendrian braid $K$}
\end{figure}

Let us take two paths $\gamma_1$ and $\gamma_2$. Let $\frakC$ be a categorification of $K$. Then the pull-backs along $\gamma_1$ and $\gamma_2$ give two exceptional collections. Suppose the exceptional collection given by $\cC=\la E_1, ..., E_n\ra$. Then the exceptional collection associated to $\gamma_2$ is a mutation associated to $\sigma$!; ``a braid mutation is a categorification of the braid''. \hspace{\fill}$\qed$
\end{example}

\section{Irregular perverse schober}
\subsection{Irregular singularities}
First let us define irregular singularity. Again, let $\bD$ be a unit disk centered at $0$ in $\bC$ and $\cO(*0)$ be the sheaf of meromorphic functions with poles at $0$. Let $\nabla$ be a connection on $\cO(*0)$, then $\nabla$ can be written as
\begin{equation}
\nabla=d+f(z)dz
\end{equation}
in the standard coordinate where $f(z)$ is a meromorphic function with poles at 0. If the order of the pole of $f$ is less than $2$, the connection $\nabla$ is {\em regular}, otherwise {\em irregular}.

One can extend the notion of the regularity to $\cD$-modules. A $\cD$-module is an $\cO$-module with the action of $\partial_z$ with Leibniz rule. In other words, it is a module over the ring $\cD=\cO\la \partial_z\ra$ where the generation is taken inside $\cE nd_\bC(\cO)$. A meromorphic connection $(\cO(*0), \nabla)$ associates a $\cD$-module $\cO(*0)$ where $\partial_z$ acts as $\nabla_{\partial_z}$. Another example is the delta function $\cD$-module $\cD\cdot \delta:=\cD/\cD\cdot z$.

Let $D^b_{\mathrm{coh}}(\cD)$ be the triangulated category of cohomologically coherent $\cD$-modules. Let $D^b_{\mathrm{rh}}(\cD,0)$ be the triangulated hull of regular meromorphic connections $(\cO(*0), \nabla)$ and the delta function $\cD$-module. Let $\mathop{\mathrm{Mod}}_{\mathrm{rh}}(\cD,0)\subset D^b_{\mathrm{rh}}(\cD,0)$ be the full subcategory spanned by objects concentrated in degree 0. Then the regular Riemann--Hilbert correspondence states an equivalence between $\mathop{\mathrm{Mod}}_{\mathrm{rh}}(\cD,0)$ and $\mathrm{Perv}(\bD, 0)$.

In the definition of $\mathop{\mathrm{Mod}}_{\mathrm{rh}}(\cD,0)$, if we replace regular meromorphic connections with irregular meromorphic connections, we obtain irregular holonomic $\cD$-modules $\mathop{\mathrm{Mod}}_{\mathrm{hol}}(\cD,0)$. In the irregular case, to state Riemann--Hilbert correspondence, we have to take a bit more care. 

A key fact is the following Hukuhara--Levelt--Turritten theorem. Let $f=\sum_k c_kz^{k/l}$ be a Puiseux series in $\bC((z^{1/l}))$. Then we set $\cE(f)$ to be a rank 1 free $\bC((z^{1/l}))$-module with the action of $\nabla:=d+df$. 
We set $\bC((z^{1/\infty})):=\bigcup_{l}\bC((z^{1/l}))$. The isomorphism class of $\cE(f)$ only depends on the class $[f]\in \bC((z^{1/\infty}))/z^{-1}\bC[[z^{1/\infty}]]$. 

\begin{theorem}[Hukuhara--Levelt--Turritten theorem]
Let $(\cO(*0)^n,\nabla)$ be a meromorphic connection. Then there exists a subset $\{f_1,...,f_m\}\subset \bC((z^{1/\infty}))/z^{-1}\bC[[z^{1/\infty}]]$ such that the ramified formal completion of $(\cO(*0)^n,\nabla)$ is isomorphic to $\bigoplus \cE(f_i)\otimes R_i$ where each $R_i$ is a regular connection.
\end{theorem}
We call the set of classes $\{ f_1,.., f_m\}\subset \bC((z^{1/\infty}))/z^{-1}\bC[[z^{1/\infty}]]$ {\em the formal type} of $(\cO(*0)^n, \nabla)$. 

Let us fix a formal type  $T:=\{f_1,..., f_m\}$ and fix a lift to a set of meromorphic functions ${\tilde{f}_1,.., \tilde{f}_n}$ (the choice of lift requires a little more care. See the example below). We draw a Legendrian knot in the following procedure~\cite{STWZ}. Let us fix a small positive number $\epsilon$. We set 
\begin{equation}
n_{i}(\theta):=\Re\lb \tilde{f_i}|_{z=\epsilon e^{i\theta}}\rb.
\end{equation}
Here $\tilde{f}_i$ is an element of the class $[f_i]$. The graph of $n_{i}(\theta)$ is living in $S^1\times \bR$. By coorientating towards $-\infty$, we get a front projection of Legendrian knot.

\begin{example}
Consider the Airy equation $\partial_t^2f-tf=0$. This equation has an irreugular singularity at $\infty$. We change the coordinate by $t=1/z$. Then the formal type of this equation is $\{\pm z^{-\frac{5}{2}}\}$.Then $n_{i}(\theta)=(-1)^{i+1}\Re \epsilon^{-\frac{3}{2}}e^{-\frac{3i}{2}\theta}$. The two $n_{i}(\theta)$ form a single multi-valued function by the monodromy. In this case, we have an immersion of a single circle as the following picture (famously first drawn by Stokes):
\begin{figure}[h]
\begin{center}
  \includegraphics[width=5cm]{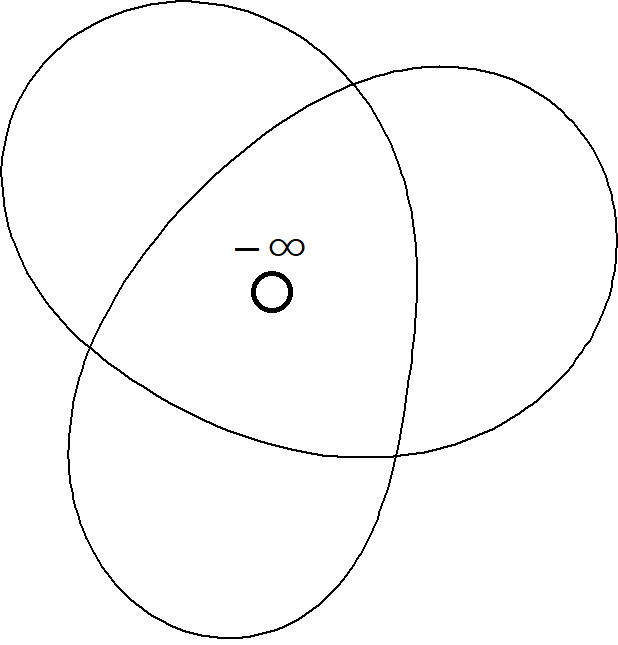}
\end{center}
\caption{Airy knot}
\end{figure}
\end{example}

In general, the picture is an immersion of some circles. We denote the Legendrian knot by $K(T)$. Let $\mathrm{Sh}^p_{L(K(T))}(S^1\times \bR)_0$ be the category of pure sheaves microsupported in $L(K(T))$ such that the stalk at $(\theta,t)\in S^1\times \bR$ with $t\ll 0$ is 0. Let $\mathrm{Mero}^T(\bD,0)$ be the category of meromorphic connections with the formal type $T$, which is a full subcategory $\mathop{\mathrm{Mod}}_{\mathrm{hol}}(\cD,0)$.
\begin{theorem}[Deligne, Malgrange, Shibuya, ..., Shende--Treumann--Williams--Zaslow~\cite{STWZ}]
There exists an equivalence
\begin{equation}
\mathrm{Mero}^T(\bD,0)\simeq \mathrm{Sh}^p_{L(K(T))}(S^1\times \bR)_0.
\end{equation}
\end{theorem}

Let $\cM$ be a holonomic $\cD$-module. The formal type of $T$ is defined by the formal type of $\cM\otimes \cO(*0)$, which is a meromorphic connection. Let $\mathop{\mathrm{Mod}}^T_{\mathrm{hol}}(\cD,0)$ be the full subcategory of $\mathop{\mathrm{Mod}}_{\mathrm{hol}}(\cD,0)$ spanned by objects of formal type $T$.

Let $\cE$ be an object of the category $\mathrm{Sh}^p_{L(K(T))}(S^1\times \bR)_0$. Let $K(0)$ be the component of $K$ corresponding to $f=0$. Now let us implicitly identify $S^1\times \bR$ with $\bD\bs\{0\}$. Recall the skeleton $L$ considered in section 1 and take a point $p\in L\cap (\bD\bs 0)$. Let $\cE_p^o$ be the microstalk over $\cE$ over $K(0)$ at $p$. We set $M$ the monodromy of $\cE^o_{p}$ around $0$. 

Let us introduce a category $\cC_{T}$ given by the following data:
\begin{enumerate}
\item Object: A pair $(\cE, V, f, g)$ where $\cE$ is an object of $\mathrm{Sh}^p_{L(K(T))}(S^1\times \bR)_0$, $V$ is a finite-dimensional $\bC$-vector space, and linear maps $f\colon \cE^o_{p}\rightarrow V$ and $g\colon V\rightarrow \cE^o_{p}$ such that $\id-f\circ g$ and $\id-g\circ f$ are invertible and $\id-g\circ f=M$.
\item Morphism: Compatible maps.
\end{enumerate}

The following theorem is stated by Malgrange~\cite{Mal} (see also~\cite{Sabbah}). We present a sketch of proof using D'Agnolo--Kashiwara's irregular Riemann--Hilbert correspondence~\cite{DagKas}.
\begin{theorem}[Irregular Beilinson theorem]
There exists an equivalence between $\cC_T$ and $\mathop{\mathrm{Mod}}^T_{\mathrm{hol}}(\cD,0)$.
\end{theorem}
\begin{proof}
We only sketch how to construct the corresponding objects. Suppose given an object in $\cC_T$. The regular Beilinson theorem (Theorem 1.1) gives us a perverse sheaf $P$ from the data of $(\cE^o_{p}, V, f, g)$. On the other hand, we have an enhanced ind-sheaf~\cite{DagKas} (or irregular $\bC$-constructible sheaf~\cite{Kuw}) over $\bD$ corresponding to $\cE$, which will be denoted by $E$. Let us take a small open disk $D$ around $0$ and consider the restriction of $E$ to $S^1=\partial D$. Let us put the perverse sheaf $P$ on $D$ with singularity on $0$ as an enhanced ind-sheaf.

As noted in \cite{DagKas2}, the restriction of $E$ to $S^1$ is precisely $\cE$ up to Legendrian isotopy. Let $U$ be the the connected component of the complement of $K(0)$ which contains $\infty$. Let $\cL$ be the local system on $U$ with monodromy $M$. Then there exists a canonical morphism $\cL\rightarrow E|_{S^1}=\cE$. By shrinking $S^1$, this gives a morphism $L\rightarrow E$ as enhanced ind-sheaves where $L$ is a local system over $(D\bs 0)\times \bR_{>0}$. Note that there also exists a canonical morphism from from $L$ as enhanced sheaves.

Take the gluing i.e. the kernel of $L\rightarrow P\oplus E$. This satisfies the irregular perversity condition~\cite{Kuw}, hence gives an object of $\mathop{\mathrm{Mod}}^T_{\mathrm{hol}}(\cD,0)$.

On the other hand, given an object $\cM$ of $\mathop{\mathrm{Mod}}^T_{\mathrm{hol}}(\cD,0)$, consider a meromorphic connection $\cM\otimes \cO(*0)$. By taking the Riemann-Hilbert image of this connection, we get an object $\cE$ of $\mathrm{Sh}^p_{L_{K(T)}}(S^1\times\bR)_0$. Again, we denote the counterpart as an enhanced ind-sheaf by $E$. Consider the exact triangle 
\begin{equation}
E\rightarrow Sol(\cM)\rightarrow Q\xrightarrow{[1]}
\end{equation}
which is the image of the exact triangle extending the morphism $\cM\rightarrow \cM\otimes \cO(*0)$ under D'Agnolo--Kashiwara functor $Sol$. Then $Q$ is supported over 0. Let $L'$ be the local system corresponding to $f=0$-part of $E$. Then there exists a morphism $E\rightarrow L'$ as enhanced sheaves. Composing this map with the extension map $Q\rightarrow E[1]$, we get a perverse sheaf as the cone of $Q[-1]\rightarrow L'$.
\end{proof}

\subsection{Irregular perverse schober}
Let us define an irregular perverse schober. For a given formal type $T:=\{f_1,..., f_n\}$, we get a Legendrian knot $K(T)$.
\begin{definition}
Suppose $0\not\in T$. A Stokes schober of the formal type $T$ is a categorification of $K(T)$ following Ansatz 2. 
\end{definition}
A Stokes schober gives a set of semi-orthogonally decomposed triangulated categories labeled by Stokes rays. The left mutation of a semi-orthogonal decomposition in this sequence is identified with the next semi-orthogonal decomposition by an equivalence. Note that walking around $0\in \bD$, we get a monodromy autoequivalence for each $\cC_i$. This set of data was originally used in Sanda--Shamoto~\cite{SS} to treat Dubrovin type conjecture (see also Example 5.10).

Recall $L$ a skeleton of $\bD$.
\begin{definition}
Suppose $0\in T$ and $f_i=0$. An irregular perverse schober of the formal type $T$ is given by the following data:
\begin{enumerate}
\item A categorification $\frakC$ of $K(T)$ following ansatz 2. Let $\cC=\la \cC_1,..., \cC_i,..., \cC_{n} \ra$ be the semi-orthogonal decomposition associated to $\fC$ along $L$.
\item A triangulated category $\cD$ and a perverse schober consisting of $\cD$ and $\cC_i$ such that the spherical twist for $\cC_i$ is the same as the monodromy autoequivalence of $\cC_i$.
\end{enumerate}
\end{definition}
The author was informed that Sanda--Shamoto obtained the same definition previously. The irregular Beilinson theorem tells us that this is actually a categorification of an irregular singularity i.e., by taking $K_0\otimes_\bZ\bC$, it gives an irregular $\cD$-module.

\begin{example}[N-spherical functors]
Consider the knot given in Figure \ref{figsphe}.
\begin{figure}[h]
\begin{center}
  \includegraphics[width=5cm]{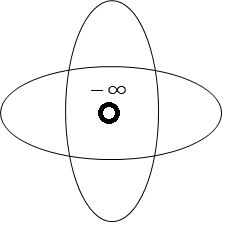}
\end{center}
\caption{Spherical functor}\label{figsphe}
\end{figure}

For example, a formal type $T=\{1/z^2, \sqrt{-1}/z^2\}$ gives the knot. By the definition, the corresponding irregular perverse schober is given by the data (here we assume the involved equivalences in Ansatz 2 are the identities): a semi-orthogonal decomposition $\la \cC_1, \cC_2\ra$ such that the mutation of $\la \cC_1, \cC_2\ra$ is 4-periodic. Recall the following theorem.
\begin{theorem}[Halpern-Leistner--Shipman~\cite{HLS}]
A four-periodic semi-orthogonal decomposition gives a spherical functor and the converse is also true.
\end{theorem}
Hence this irregular perverse schober gives a spherical functor. Note that there is no $\cD$ since $0\not\in T$.

One can also consider the following knot where the number of crossing is $2N$.
\begin{figure}[h]
\begin{center}
  \includegraphics[width=8cm]{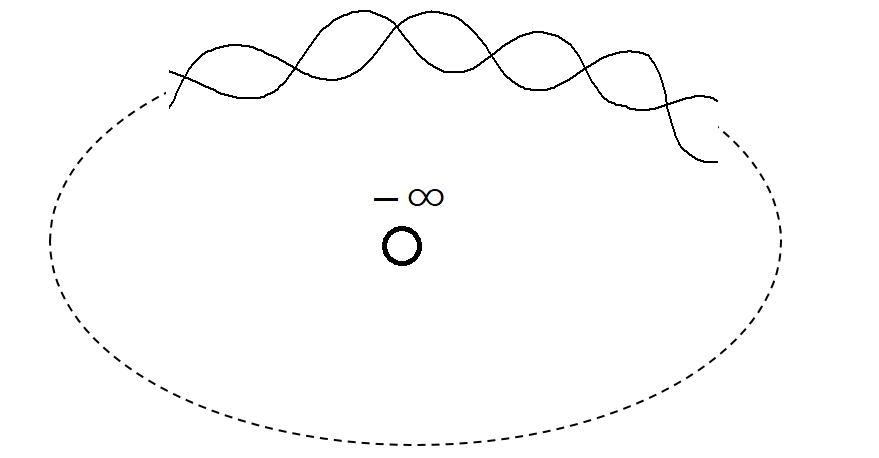}
\end{center}
\caption{$N$-Spherical functor}
\end{figure}

By the same argument, this associates an N-spherical functor in the sense of Dyckerhoff--Kapranov--Schechtman~\cite{DKS}. \hspace{\fill}$\qed$
\end{example}

\begin{example}[Quantum $\cD$-modules]
The relation between irregular singularities and semi-orthogonal decompositions has been observed in the context of Dubrovin's conjecture. In particular, the relation between mutation of SOD and Stokes structure was studied and conjectured by Sanda--Shamoto~\cite{SS}. In our language, their conjecture can be rephrased as follows:
\begin{conjecture}[(a part of) Sanda--Shamoto's Dubrovin conjecture~\cite{SS}]
Let $X$ be a Fano manifold. There exists an irregular perverse schober whose nearby cycle is $D^b(X)$ and the Hochschild decategorification gives a Stokes data which is irregular Riemann--Hilbert image of the quantum $\cD$-module of $X$ around $0\in \bP^1_\hbar$. 
\end{conjecture}

Irregular singularities of quantum $\cD$-modules appear not only in $\hbar$-directions but also Kaehler directions. In the work announced by Iritani, irregular singularities of quantum $\cD$-module are observed in the situation of toric flips. By the philosophy of ``discrepant resolution conjecture'', this should correspond to semi-orthogonal decompositions of the derived category of coherent sheaves and should form an irregular perverse schober. The B-model consideration of this subject will be explored in a work in progress joint with Will Donovan. \hspace{\fill}$\qed$
\end{example}

\section*{Acknowledgment}
This note is based on my talk in ``Categorical and analytic invariants in algebraic geometry IV'' held in Hokkaido 2018. The author thanks the organizers and participants for their interests.

The author admires 75th birthday of Professor Kyoji Saito and thanks his enormous contribution to settle IPMU mathematics, where the author works as a postdoc in a comfortable atmosphere. Also, the author thanks Saito sensei for organizing reading seminars in 2016 where the author learned about irregular singularities and for having daily bento lunches together.

The author also thanks Will Donovan, Mikhail Kapranov, and Vivek Shende for discussion about this topic and Will Donovan and Fumihiko Sanda for comments on an early draft. Especially, the first idea of this note came from Mikhail's talk on N-spherical functors in Osaka 2017.

This work was supported by World Premier International Research Center Initiative (WPI), MEXT, Japan and JSPS KAKENHI Grant Number JP18K13405.

\bibliographystyle{amsalpha}
\bibliography{CategorificationofLegendrianknotsarxiv2.bbl}

\providecommand{\bysame}{\leavevmode\hbox to3em{\hrulefill}\thinspace}
\providecommand{\MR}{\relax\ifhmode\unskip\space\fi MR }
\providecommand{\MRhref}[2]{%
  \href{http://www.ams.org/mathscinet-getitem?mr=#1}{#2}
}
\providecommand{\href}[2]{#2}
\begin{thebibliography}{BKS18}

\bibitem[AL17]{AL}
Rina Anno and Timothy Logvinenko, \emph{Spherical {DG}-functors}, J. Eur. Math.
  Soc. (JEMS) \textbf{19} (2017), no.~9, 2577--2656. \MR{3692883}

\bibitem[Be{\u{\i}}87]{Beilinson}
A.~A. Be{\u{\i}}linson, \emph{How to glue perverse sheaves}, {$K$}-theory,
  arithmetic and geometry ({M}oscow, 1984--1986), Lecture Notes in Math., vol.
  1289, Springer, Berlin, 1987, pp.~42--51. \MR{923134}

\bibitem[BKS18]{BKS}
Alexey Bondal, Mikhail Kapranov, and Vadim Schechtman, \emph{Perverse schobers
  and birational geometry}, Selecta Math. (N.S.) \textbf{24} (2018), no.~1,
  85--143. \MR{3769727}

\bibitem[DK]{DonovanKuwagaki}
Will Donovan and Tatsuki Kuwagaki, \emph{Mirror symmetry for perverse schobers
  from birational geometry}, in preparation.

\bibitem[DK16]{DagKas}
Andrea D'Agnolo and Masaki Kashiwara, \emph{Riemann-{H}ilbert correspondence
  for holonomic {D}-modules}, Publ. Math. Inst. Hautes \'{E}tudes Sci.
  \textbf{123} (2016), 69--197. \MR{3502097}

\bibitem[DK18]{DagKas2}
\bysame, \emph{A microlocal approach to the enhanced {F}ourier-{S}ato transform
  in dimension one}, Adv. Math. \textbf{339} (2018), 1--59. \MR{3866893}

\bibitem[DKS]{DKS}
Tobias Dyckerhoff, Mikhail Kapranov, and Vadim Schechtman, \emph{N-spherical
  functors}, in preparation.

\bibitem[Dona]{Donovan1}
Will Donovan, \emph{Perverse schobers and wall crossing}, arXiv:1801.05319.

\bibitem[Donb]{Donovan2}
\bysame, \emph{Perverse schobers on riemann surfaces: constructions and
  examples}, arXiv:1703.00592.

\bibitem[HK]{HarKat}
Andrew Harder and Ludmil Katzarkov, \emph{Perverse sheaves of categories and
  some applications}, arXiv:1708.01181.

\bibitem[HLS16]{HLS}
Daniel Halpern-Leistner and Ian Shipman, \emph{Autoequivalences of derived
  categories via geometric invariant theory}, Adv. Math. \textbf{303} (2016),
  1264--1299. \MR{3552550}

\bibitem[KS]{KapS}
Mikhail Kapranov and Vadim Schechtman, \emph{Perverse schobers},
  arXiv:1411.2772.

\bibitem[KS90]{KS}
Masaki Kashiwara and Pierre Schapira, \emph{Sheaves on manifolds}, Grundlehren
  der Mathematischen Wissenschaften [Fundamental Principles of Mathematical
  Sciences], vol. 292, Springer-Verlag, Berlin, 1990, With a chapter in French
  by Christian Houzel. \MR{1074006}

\bibitem[Kuw]{Kuw}
Tatsuki Kuwagaki, \emph{Irregular perverse sheaves}, arXiv:1808.02760.

\bibitem[Mal91]{Mal}
B.~Malgrange, \emph{\'{E}quations diff\'{e}rentielles \`a coefficients
  polynomiaux}, Progress in Mathematics, vol.~96, Birkh\"{a}user Boston, Inc.,
  Boston, MA, 1991. \MR{1117227}

\bibitem[Nad]{Nadler}
David Nadler, \emph{Mirror symmetry for the landau-ginzburg a-model
  $m=\mathbb{C}^n$, $w=z_1 \dots z_n$}, arXiv:1601.02977.

\bibitem[Sab13]{Sabbah}
Claude Sabbah, \emph{Introduction to {S}tokes structures}, Lecture Notes in
  Mathematics, vol. 2060, Springer, Heidelberg, 2013. \MR{2978128}

\bibitem[SS]{SS}
Fumihiko Sanda and Yota Shamoto, \emph{An analogue of dubrovin's conjecture},
  arXiv:1705.05989.

\bibitem[STWZ]{STWZ}
Vivek Shende, David Treumann, Harold Williams, and Eric Zaslow, \emph{Perverse
  schobers}, arXiv:1512.08942.

\bibitem[STZ17]{STZ}
Vivek Shende, David Treumann, and Eric Zaslow, \emph{Legendrian knots and
  constructible sheaves}, Invent. Math. \textbf{207} (2017), no.~3, 1031--1133.
  \MR{3608288}

\end{thebibliography}

\noindent

\end{document}